\scrollmode
\documentclass{amsart}
\usepackage{amssymb}
\usepackage[all,arc]{xy}
\usepackage{graphicx}
\usepackage{mathrsfs}
\usepackage{xcolor}

\theoremstyle{plain}
\newtheorem{prop}{Proposition}
\newtheorem{thm}[prop]{Theorem}
\newtheorem{cor}[prop]{Corollary}

\newtheorem{fact}[prop]{Fact}

\newtheorem{ques}{Question}

\newtheorem*{thmA}{Theorem A}
\newtheorem*{theoB}{Theorem B}
\newtheorem*{theoC}{Theorem C}

\theoremstyle{definition}
\newtheorem*{defi}{Definition}

\theoremstyle{remark}
\newtheorem{rem}[prop]{Remark}
\newtheorem{example}{Example}

\numberwithin{prop}{section}
\numberwithin{example}{section} 
\numberwithin{exer}{subsection}
\numberwithin{claim}{prop}
\numberwithin{step}{prop}
\numberwithin{equation}{section}

\newcommand{\N}{\mathbb{N}}
\newcommand{\Z}{\mathbb{Z}}

\newcommand{\C}{\mathbb{C}}
\newcommand{\R}{\mathbb{R}}

\newcommand{\cuD}{\mathscr{D}}
\newcommand{\cuF}{\mathscr{F}}

\newcommand{\baX}{\bar{X}}

\newcommand{\tiX}{\widetilde{X}}
\newcommand{\der}{\partial}
\newcommand{\tider}{\widetilde{\der}}

\newcommand{\iid}{\mathrm{id}}

\newcommand{\eud}{\mathfrak{d}}

\newcommand{\caO}{\mathcal{O}}
\DeclareMathOperator{\hotimes}{\widehat{\otimes}}
\newcommand{\be}{\mathbf{e}}

\newcommand{\op}{\mathrm{op}}

\newcommand{\caA}{\mathscr{A}}
\newcommand{\whE}{\widehat{E}}
\newcommand{\caR}{\mathscr{R}}

\newcommand{\euS}{\mathscr{S}}
\newcommand{\euP}{\mathscr{P}}
\newcommand{\uomega}{\underline{\omega}}
\newcommand{\ueta}{\underline{\eta}}
\newcommand{\wtM}{\widetilde{\partial}\moM}

\newcommand{\bchi}{\overline{\chi}}
\newcommand{\Bor}{\mathrm{Bor}}

\newcommand{\moM}{\mathscr{M}}
\newcommand{\moQ}{\mathscr{Q}}
\newcommand{\bmoM}{\bar{\moM}}
\newcommand{\bmoQ}{\bar{\moQ}}
\newcommand{\tilpa}{\widetilde{\partial}}
\newcommand{\cl}{\mathrm{cl}}

\newcommand{\euF}{\mathscr{F}}
\newcommand{\beuF}{\bar{\euF}}
\newcommand{\caF}{\mathcal{F}}
\newcommand{\ab}{\mathrm{ab}}
\newcommand{\bphi}{\bar{\phi}}

\newcommand{\ccB}{\mathcal{B}}

\newcommand{\diam}{\mathrm{diam}}
\newcommand{\card}{\mathrm{card}}
\newcommand{\image}{\mathrm{im}}

\newcommand{\argu}{\hbox to 7truept{\hrulefill}}

\newcommand{\cC}{\mathrm{C}}
\newcommand{\bcC}{\widetilde{\mathrm{C}}}
\newcommand{\dC}{\mbox{\reflectbox{\textup{C}}}}
\newcommand{\caB}{\mathscr{B}}
\newcommand{\caT}{\mathscr{T}}
\newcommand{\caN}{\mathscr{N}}
\newcommand{\caS}{\mathscr{S}}

\newcommand{\eps}{\varepsilon}

\begin{document}
\title{Monoids, their boundaries, fractals and $C^\ast$-algebras}
\author{Giulia dal Verme and Thomas Weigel}

\address{G. dal Verme,\ Th. Weigel\\ Dipartimento di Matematica e Applicazioni\\
Universit\`a degli Studi di Milano-Bicocca\\
Ed.~U5, Via R.Cozzi 55\\
20125 Milano, Italy}
\email{giulia.dalverme@unimib.it}
\email{thomas.weigel@unimib.it}


\begin{abstract}
In this note we establish some connections between the theory
of self-similar fractals in the sense of John E. Hutchinson (cf. \cite{hutch:frac}), and the theory of boundary quotients
of $C^\star$-algebras associated to monoids. Although we must leave several important questions open,
we show that the existence of  self-similar $\moM$-fractals for a given monoid $\moM$, gives rise to examples of $C^\star$-algebras \eqref{eq:Cstarfrac} generalizing the boundary quotients $\partial C^\ast_\lambda(\moM)$ discussed by X.~Li in 
\cite[\S 7,\, p.~71]{Li:sgrp}.
The starting point for our investigations is the observation that the
universal boundary of a finitely $1$-generated monoid carries naturally two
topologies. The fine topology plays a prominent role in the construction of these
boundary quotients. On the other hand, the cone topology can be used to define canonical measures on the attractor of an $\moM$-fractal provided $\moM$ is finitely $1$-generated.
\end{abstract}

\date{\today}


\maketitle

\section{Introduction}
\label{s:intro}
On a monoid $\moM$ (=semigroup with unit $1_{\moM}$)  there is naturally defined a
reflexive and transitive relation ``$\preceq$'', i.e.,  for $\omega,\tau\in\moM$ one defines
$\omega\preceq\tau$  
 if, and only if, there exists $\sigma\in\moM$ satisfying
$\omega=\tau\cdot\sigma$. In particular, one may consider $(\moM,\preceq)$ as a 
{\emph{partially ordered set}}. Moreover, 
if $\moM$ is $\N_0$-graded,
then $(\moM,\preceq)$ is a (noetherian) partially ordered set (see Corollary~\ref{cor:poset}).
Such a poset has a poset completion $i_{\moM}\colon \moM\longrightarrow\bmoM$
(see \S~\ref{ss:compl}), and one defines the {\emph{universal boundary}}
 $\partial\moM$ of $\moM$ by
\begin{equation}
\label{eq:tilpadef}
\partial\moM=(\bmoM\setminus\image(i_\moM))/\!\approx,
\end{equation}
where $\approx$ is the equivalence relation induced  by "$\preceq$"
on $\bmoM\setminus\image(i_\moM)$ (see \S~\ref{ss:compl}).
For several reasons (cf. Theorem A, Theorem B, Theorem C)
one may consider $\partial\moM$ as the natural boundary associated to the monoid $\moM$.
However, it is less clear what topology one should consider. Apart from the cone topology $\caT_c(\bmoM)$ there is another
potentially finer topology $\caT_f(\bmoM)$ which will be called the {\emph{fine topology}}
on $\partial\moM$ (cf. \S~\ref{ss:finetop}), i.e., the identity 
\begin{equation}
\label{eq:idtop}
\iid_{\partial\moM}\colon (\partial\moM, \caT_f(\bmoM))\longrightarrow 
(\partial\moM, \caT_c(\bmoM))
\end{equation} 
is a continuous map. The monoid $\moM$ will be said to be {\emph{$\caT$-regular}}, if 
\eqref{eq:idtop} is a homeomorphism. E.g., finitely generated free monoids are 
$\caT$-regular (cf. Proposition~\ref{prop:T2}, \S~\ref{ss:boundT}). The universal boundary 
$\eth\moM=(\partial\moM,\caT_f(\bmoM)$ with the fine topology can be identified with the
Laca boundary $\whE(\moM)$ of the monoid $\moM$.
This topological space plays an essential role for defining
boundary quotients of $C^\ast$-algebras associated to monoids (cf. \cite[\S~7]{Li:sgrp},
\cite{Lietal:Cstar}). Indeed one has the following (cf. Theorem~\ref{thm:princ}).

\begin{thmA}
The map $\bchi_{\cdot}\colon (\partial \moM,\caT_f(\bmoM))\longrightarrow \whE(\moM)$ 
defined by \eqref{eq:map2} is a homeomorphism.
\end{thmA}

\begin{rem}
\label{rem:funct}
(a) By Theorem A, the topological space $(\partial\moM,\caT_f(\bmoM))$ is
totally-disconnected and compact. On the other hand, in general one can only show
that $(\partial\moM,\caT_c(\bmoM))$ is a $T_0$-space. Indeed, if 
$(\partial\moM,\caT_c(\bmoM))$ happens to be Hausdorff, then \eqref{eq:idtop}
is necessarily a homeomorphism and $\moM$ is $\caT$-regular (cf. Proposition~\ref{prop:T2}). 
So from this perspective considering $\caT_f(\bmoM)$ as the natural topology would have
many advantages. However, there are also indications supporting the idea
to consider $\caT_c(\bmoM)$ as the natural topology.

\noindent
(b) If $\phi\colon\moQ\longrightarrow\moM$ is a surjective graded homomorphism of finitely $1$-generated monoids,
then, by construction, $\phi$ induces a surjective, continuous and open map
\begin{equation}
\label{eq:funct1}
\partial\bphi\colon (\partial\moQ,\tau_c(\bmoQ))\longrightarrow(\partial\moM,\tau_c(\bmoM))
\end{equation}
(cf. Proposition~\ref{prop:homs}). This property can be used to establish the following.

\begin{theoB}
Let $\moM$ be a finitely $1$-generated  $\N_0$-graded monoid.
Then $\partial\moM$ carries naturally a Borel probability meausure
\begin{equation}
\label{eq:bor}
\mu_\moM\colon\Bor(\partial\moM)\longrightarrow\R^+\cup\{\infty\}.
\end{equation}
\end{theoB}

On the other hand the induced mapping $\phi_{\whE}$ is given by a map
\begin{equation}
\label{eq:funct2}
\phi_{\whE}\colon \whE(\moM),\longrightarrow\whE(\moQ),
\end{equation}
(cf. Proposition~\ref{prop:hom}). Hence for the purpose of constructing Borel measures the fine topology seems to be inappropriate.
\end{rem}

Theorem B can be used to define the $C^\ast$-algebra 
\begin{equation}
\label{eq:Cstarmon}
C^\ast(\moM,\mu_\moM)=\langle\, \beta_\omega,\,\beta_\omega^\ast\mid \omega\in\moM\,
\rangle
\subseteq\ccB(L^2(\partial\moM,\C,\mu_{\moM}))
\end{equation}
for every finitely $1$-generated $\N_0$-graded monoid $\moM$,
where $\beta(\omega)$ is the mapping induced by left multiplication with $\omega$
(cf. \S~\ref{ss:caT}). We will show by explicit calculation that for the monoid $\euF_n$, freely generated
by a set of cardinality $n$, the $C^\ast$-algebra
$C^\ast(\euF_n,\mu_{\euF_n})$ coincides with the Cuntz algebra $\caO_n$
(cf. Proposition~\ref{prop:cuntz}),
while for the right-angled Artin monoid $\moM^\Gamma$ associated to the
finite graph $\Gamma$, $C^\ast(\moM^\Gamma,\mu_{\moM^\Gamma})$
coincides with the boundary quotients introduced by Crisp and Laca
in \cite{CL07} (cf. \S~\ref{ss:raam}).
Nevertheless the following more general question remains unanswered.

\begin{ques}
Let $\moM$ be a finitely $1$-generated $\N_0$-graded monoid with the left cancellation property. Then
$C^\ast(\moM,\mu_\moM)$ coincides with the boundary quotient $\partial C_\lambda(\moM)$
defined
by X.~Li in \cite[\S 7]{Li:sgrp}.
\end{ques}

From now on we will assume that the $\N_0$-graded monoid $\moM=\bigcup_{k\in\N_0}\moM_k$ is finitely $1$-generated.
In the context of self-similar fractals in the sense of 
John E. Hutchinson (cf. \cite{hutch:frac}) it will be more natural
to endow $\partial\moM$ with the cone topology.
Let $(X,d)$ be a complete metric space with a left $\moM$-action 
$\alpha\colon\moM\longrightarrow C(X,X)$
by continuous maps. Such a presentation will be said to be {\em{contracting}},
if there exists a positive real number $\delta<1$ such that
\begin{equation}
\label{eq:cont}
d(\alpha(s)(x),\alpha(s)(y))\leq \delta\cdot d(x,y),
\end{equation}
for all $x,y\in X$, $s\in\moM_1$ (cf. \cite[\S~2.2]{hutch:frac}).
For such a metric space $(X,d)$
there exists a unique compact subset $K\subseteq X$ such that
\begin{itemize}
\item[(1)] $K=\bigcup_{s\in\moM_1} \alpha(s)(K)$,
\item[(2)] $K=\cl(\{\,Fix(\alpha(t))\mid t\in \moM\,\}\subseteq X$.
\end{itemize}
Obviously, by definition every map $\alpha(t)\in C(X,X)$ is contracting, and thus
has a unique fixed point $x_t\in X$. For short we call $K=K(\alpha)\subset X$ the {\em{attractor}}
of the representation $\alpha$. One has the following (cf. Proposition~\ref{prop:uniprop}).

\begin{theoC}
Let $\moM=\bigcup_{k\in\N_0}\moM_k$ be a finitely $1$-generated $\N_0$-graded
monoid, let $(X,d)$ be a complete metric space and let
$\alpha\colon\moM\longrightarrow C(X,X)$ be a contracting representation of $\moM$.
Then for any point $x\in X$, $\alpha$ induces a continuous map
\begin{equation*}
\label{eq:attmap}
\kappa_x\colon\partial\moM\longrightarrow K(\alpha).
\end{equation*}
Moreover, if $\moM$ is $\caT$-regular, then $\kappa_x$ is surjective.
\end{theoC}

Under the general hypothesis of Theorem C we do not know whether the topological space  $(\partial\moM, \caT_c(\bmoM))$
is necessarily compact (see Question~\ref{ques:comp}).
However, in case that it is, then
Theorem~C offers the interpretation that
$(\partial\moM, \caT_c(\bmoM))$ may be considered as a kind of 
{\em{universal attractor of the
finitely $1$-generated $\N_0$-graded $\caT$-regular monoid $\moM$}. }

\begin{rem}
\label{rem:fracCstar}
Let $\moM$ be a finitely $1$-generated monoid.
Then $\partial\moM$ carries canonically a probability measure $\mu_\moM$
(cf. \S~\ref{ss:caT}).
Thus, by Theorem C, the attractor of the $\moM$-fractal $((X,d),\alpha)$
carries a type of \emph{contact probability measure} $\mu_x=\mu_\circ^{\kappa_x}$
for every point $x\in X$, which is given by
\begin{equation}
\label{eq:kont}
\mu_x(B)=\mu_{\moM}(\kappa_x^{-1}(B)),\qquad B\in\Bor(K).
\end{equation}

By (1), the monoid $\moM$ is acting on $K$, and thus also on
$L^2(K,\C,\mu_x)$ by bounded linear operators $\gamma(\omega)$, $\omega\in\moM$
(cf. \S~\ref{ss:Mfrac})
This defines a $C^\ast$-algebra (cf. \S~\ref{ss:Mfrac})
\begin{equation}
\label{eq:Cstarfrac}
C^\ast(\moM,X,d,\mu_x)=\langle\,\gamma(\omega),\,\gamma(\omega)^\ast\mid\omega\in\moM
\rangle\subseteq\ccB(L^2(K,\C,\mu)).
\end{equation}
\end{rem}

The action of $\partial\moM$ on $K$ factors through an $\tilpa\moM$-action,
and $\tilpa\moM$ is in general different from $(\partial\moM,\caT_c(\bmoM))$
(cf. Remark~\ref{rem:tilno}). Hence it seems plausible that the $C^\ast$-algebras
$C^\ast(\moM,X,d,\mu_x)$ are in general different from $C^\ast(\moM,\mu_\moM)$.
However further investigations seem to be necessary in order to clarify this 
point in more detail.


\section{Posets and their boundaries}
\label{s:poset}
A poset (or partially ordered set)  is a set $X$ together with a 
reflexive and transitive relation $\preceq\colon X\times X\to\{t,f\}$ 
with the property that for all $x,y\in X$
satisfying $x\preceq y$ and $y\preceq x$ follows that $x=y$.
By $\N=\{1,2,\ldots\}$ we denote the set of positive integers,
and by $\N_0=\{0,1,2,\ldots\}$ we denote the set of non-negative integers, i.e.,
$\N_0$ is a commutative monoid.


\subsection{Cones, cocones and intervalls}
\label{ss:cones}
For a poset $(X,\preceq)$ and $\tau, \omega\in X$ the set
\begin{align}
\cC_\omega&=\{\,x\in X\mid x\preceq\omega\,\}\label{eq:cone}\\
\intertext{will be called the {\em cone} defined by $\omega$, and}
\dC_\tau&=\{\,y\in X\mid y\succeq\tau\,\}\label{eq:cocone}\\
\intertext{the {\em cocone} defined by $\tau$. 
For $\tau\preceq\omega$ the set}
[\tau,\omega]&=\dC_\tau\cap\cC_\omega=\{\,x\in X\mid
\tau\preceq x\preceq\omega\,\}
\end{align}
is called the {\em closed intervall} from $\tau$ to $\omega$,
i.e., $[\omega,\omega]=\{\omega\}$.
The poset $(X,\preceq)$ is said to be 
{\em noetherian},
if $\card(\dC_\tau)<\infty$ for all $\tau\in X$.


\subsection{Complete posets}
\label{ss:compo}
For a poset $(X,\preceq)$ let 
\begin{equation}
\label{eq:dec}
\cuD(\N,X,\preceq)=\{\,f\in\cuF(\N,X)\mid \forall n,m\in\N:\ n\leq m\,\Longrightarrow\ 
f(n)\succeq f(m)\,\}
\end{equation}
denote the set of decreasing functions which we will - if necessary - identify with the set of decreasing sequences. A poset $(X,\preceq)$ is said to be {\it complete}, if for all 
$f\in\cuD(\N,X,\preceq)$ there exists an element $z\in X$ such that
\begin{itemize}
\item[(CP$_1$)] $f(n)\succeq z$ for all $n\in\N$, and
\item[(CP$_2$)] if $y\in X$ satisfies $f(n)\succeq y$ for all $n\in\N$, then $z\succeq y$.
\end{itemize}
Note that - if it exists - $z\in X$ is the unique element satisfying (i) and (ii) for 
$f\in\cuD(\N,X,\preceq)$.
As usually, $z=\min(f)$ is called
the {\em minimum} of $f\in\cuD(\N,X,\preceq)$.


\subsection{The poset completion of a poset}
\label{ss:compl}
Let $(X,\preceq)$ be a poset. 
For $u,v \in \cuD(\N,X,\preceq)$ put
\begin{equation}
\label{eq:rel}
u\preceq v\ \Longleftrightarrow \forall n\in\N\ \exists k_n\in\N\colon\ u(k_n)\preceq v(n),
\end{equation}
and put
\begin{equation}
\label{eq:rel2}
u\sim v \Longleftrightarrow \,\, \big[\,(u\preceq v \wedge v\preceq u)
\vee \big(v\preceq u \wedge v=c_m, m=min(u)\big)\,\big],
\end{equation}
where  $c_z\in \cuD(\N,X,\preceq)$, $z\in X$, is given by $c_z(n)=z$ for all $n\in\N$.

Let $\approx$ be the equivalence relation generated by $\sim$ and put $\baX=\cuD(\N,X,\preceq)/\!\approx$.
Then the following properties hold for $(\baX,\preceq)$.

\begin{prop}
\label{prop:comp}
Let $(X,\preceq)$ be a poset.
\begin{itemize}
\item[(a)] The relation $\preceq$ defined by \eqref{eq:rel} is reflexive and transitive.
\item[(b)] For any strictly increasing function $\alpha\colon\N\to\N$ and $u\in\cuD(\N,X,\preceq)$ 
one has $u\approx u\circ\alpha$.
\item[(c)] Define for $[u], [v]\in \baX$ that $[u]\preceq[v]$ if,
and only if, $u\preceq v$. Then $(\baX,\preceq)$ is a poset.
\item[(d)] $(\baX,\preceq)$ is complete.
\end{itemize}
\end{prop}

\begin{proof}
(a) The relation $\preceq$ is obviously reflexive.
Let $u,v,w\in\cuD(\N,X,\preceq)$, $u\preceq v$, $v\preceq w$.
Then for all $n\in\N$ there exists $h_n, k_n\in\N$ such that 
$u(h_n)\preceq v(k_n)\preceq w(n)$. Thus, $u \preceq w$.

\noindent
(b) Let $u\in\cuD(\N,X,\preceq)$ and let $\alpha\colon\N\to\N$
be a strictly increasing function.
Let $m<n$, $m,n\in\N$. Since $\alpha$ is strictly increasing, $\alpha(m)<\alpha(n)$.
Then there exist $m_0,n_0\in\N$ such that $m_0\le\alpha(m)<\alpha(n)\le n_0$.
Then one has $u(m_0)\succeq u(\alpha(m))\succeq u(\alpha(n))\succeq u(n_0)$.
Thus $u\preceq u\circ\alpha$ and $u\circ\alpha\preceq u$, proving that $u\approx u\circ\alpha$.

\noindent
(c) Let $[u], [v]\in\baX$, $[u]\preceq [v]$ and $[v]\preceq [u]$.
Then, by definition, $u\preceq v$ and $v\preceq u$, and thus $u\approx v$, i.e.,
$[u]=[v]$.

\noindent
(d) Let $\{u^k\}_{k\in\N}\in\cuD(\N,\baX,\preceq)$, i.e., 
$u^k\in\baX$ for all $k\in\N$.
Then one has $u^1\succeq u^2\succeq\dots$ by definition.
Since each $u^k\in\cuD(\N,X,\preceq)$, 
one has $u^k(n)\succeq u^k(m)$ for all $n\le m$, $m,n\in\N$.
We define $v\in\cuD(\N,X,\preceq)$ by $v(n)=u^n(n)$, $n\in\N$.
Then $[v]\in\baX$ is the minimum of $\{u^k\}_{k\in\N}$.
This yields the claim.
\end{proof}

Assigning every element $x\in X$ the equivalence class containing the
constant function $c_x\in\cuD(\N,X,\preceq)$ satisfying
$c_x(n)=x$ for all $n\in\N$ yields a strictly increasing mapping of posets
$\iota_X\colon X\to\baX$. 
From now on $(X,\preceq)$ will be considered as a sub-poset of $(\baX,\preceq)$.
The poset $(\baX,\preceq)$ will be called the 
{\em poset completion of $(X,\preceq)$}.
The following fact is straightforward.

\begin{fact}
\label{fact:compl}
The map $\iota_X$ is a bijection if, and only if, $(X,\preceq)$ is complete.
\end{fact}

\begin{example}
Let $X=\N\sqcup\{\infty\}$ and define $n\preceq m$ if, and only if, $n\geq m$,
where ``$\geq$'' denotes the natural order relation.
Then the poset $(X,\preceq)$ is complete and $\baX=X$.
\end{example}


\subsection{The universal boundary of a poset}
\label{ss:bound}
For a poset $(X,\preceq)$ the poset
 $\partial X=\baX\setminus \image(i_X)$ will be called
the {\em universal boundary} of the poset $(X,\preceq)$. 
From now on we use the notation $x\succ y$ as a short form for 
$x\succeq y$ and $x\not= y$. A function $f\colon\N\to X$
will be said to be strictly decreasing, if $f(n+1)\prec f(n)$ for all $n\in\N$.
The following fact will turn out to be useful.

\begin{fact}
\label{fact:strict}
Let $f\in\cuD(\N,X,\preceq)$ be a decreasing function such that $[f]\in\partial X$.
Then there exists a strictly decreasing function $h\in\cuD(\N,X,\preceq)$ 
such that $f\approx h$, i.e., $[f]=[h]$.
\end{fact}

\begin{proof}
By hypothesis, $J=\image(f)$ is an infinite set. In particular,  the set
$\Omega=\{\,\min(f^{-1}(\{j\})\mid j\in J\,\}$ is an infinite and unbounded subset of $\N$.
Let $e\colon\N\to\Omega$ be the enumeration function of $\Omega$, i.e.,
$e(1)=\min(\Omega)$, and recursively one has 
$e(k+1)=\min(\Omega\setminus\{e(1),\ldots, e_k\})$.
Then, by construction, $h=f\circ e$ is strictly decreasing, and, by
Proposition~\ref{prop:comp}(b),
one has $f\approx h$, and hence the claim.
\end{proof}

\begin{fact}
\label{fact:noeth}
Let $(X,\preceq)$ be a noetherian poset, and let $(\baX,\preceq)$
be its completion. Then for all $\tau\in X$ one has $\dC_\tau(\baX)\subseteq X$.
In particular, $\dC_\tau(\baX)=\dC_\tau(X)$,
where the cocones are taken in the respective posets.
\end{fact}

\begin{example}
Let $X=A\sqcup B$, where $A,B=\Z$ and define
\begin{equation}
n\preceq m\Longleftrightarrow\, \Big(((n,m\in A \vee n,m \in B) \wedge n\le m) \vee (n\in A \wedge m\in B)\Big),
\end{equation}
where ``$\leq$'' denotes the natural order relation on $\Z$.
Then $(X,\preceq)$ is a poset and its completion is given by 
$\baX=\Z\sqcup\{-\infty\}\, \bigsqcup \Z\sqcup\{-\infty\}$.
For $n\in A$, one has $\dC_n(X)\neq \dC_n(\baX)$, 
since $-\infty\in B$ is in $\dC_n(\baX)$, but not in $\dC_n(X)$.
\end{example}


\subsection{The cone topology}
\label{ss:cone}
Let $(X,\preceq)$ be a poset, and let
$(\baX,\preceq)$ denote its completion.
For $\tau,\omega\in X$ let
\begin{equation}
\label{eq:sect}
S(\tau,\omega)=\{\,x\in X\mid x\preceq\tau\ \wedge\ x\preceq\omega\,\}.
\end{equation}
By transitivity,
\begin{equation}
\label{eq:cone1}
\cC_\tau(\baX)\cap\cC_\omega(\baX)=\bigcup_{z\in S(\tau,\omega)}\cC_z(\baX).
\end{equation}
In particular,
\begin{equation}
\label{eq:cone2}
\caB_c(\baX)=\big\{\,\{x\}\mid x\in X\,\big\}\cup\big\{\,\cC_\omega(\baX)\mid
\omega\in X\,\big\}
\end{equation}
is a base of a topology $\caT_c(\baX)$ - the {\em cone topology} - on $\baX$. By construction,
the subspace $X$ is discrete and open, and the subspace $\der X$ is closed.

For $\omega\in\baX$ let $\caN_c(\omega)$ denote the set of all open 
neighborhoods of $\omega$ with respect to the cone-topology, and put 
$\caS(\omega)=\bigcap_{U\in\caN_c(\omega)} U$. Then, by construction, one has
$\caS(\omega)=\{\omega\}$ for $\omega\in X$,
and $\caS(\omega)=\cC_\omega(\baX)$ for $\omega\in\der X$.
This implies the following.

\begin{prop}
\label{prop:coneT}
Let $(X,\preceq)$ be a poset, and let $(\baX,\preceq)$ denote its completion.
Then $(\baX,\caT_c(\baX))$ is a $T_0$-space (or Kolmogorov space).
\end{prop}
\begin{proof}
Let $\tau,\omega\in\baX$, $\tau\not=\omega$.
If either $\tau\in X$ or $\omega\in X$, then either $\{\tau\}$ or
$\{\omega\}$ is an open set. So we may assume that $\tau,\omega\in\der X$.
As $\caS(\omega)=\cC_\omega(\baX)$, either there exists $U\in\caN_c(\omega)$,
$\tau\not\in U$, or $\tau\preceq\omega$. By changing the role of $\omega$ and $\tau$,
either there exists $V\in\caN_c(\tau)$,
$\omega\not\in V$, or $\omega\preceq\tau$. Since $\tau\preceq\omega$
and $\omega\preceq\tau$ is impossible, this yields the claim.
\end{proof}

\subsection{The fine topology}
\label{ss:finetop}
For a partially ordered set $(X,\preceq)$ let 
\begin{equation}
\label{eq:finetop1}
\euS=\{\,\{\tau\},\,C_\tau(\baX),\,C_\tau(\baX)^C\mid\tau\in X\,\}
\end{equation}
denote the set of all 1-elementary sets in $X$, and all cones and their complements in $X$. Then $\euS$ is a subbasis of a topology $\caT_f(\baX)$ on $\baX$ which will  call the {\em{fine topology}} on $\baX$.
In particular, the set $\Omega=\{\,X=\bigcap_{1\leq j\leq r}X_j\mid X_1,\ldots,X_r\in\euS\,\}$
is a base of the topology $\caT_f(\baX)$. By definition, this topology has the following
properties.

\begin{fact} 
\label{fact:fintop}
Let $(X,\preceq)$ be a partially ordered set. Then
\begin{itemize}
\item[(a)] $(\baX,\caT_f(\baX))$ is a $T_2$-space (or Hausdorff space).
\item[(b)] $\caT_c(\baX)\subseteq\caT_f(\baX)$.
\end{itemize} 
\end{fact}


\subsection{The $\sim$-boundary}
\label{ss:sbound}
There is another type of boundary for a poset, the $\sim$-boundary, which
seems to be relevant for the study of fractals.
Let $(X,\preceq)$ be a noetherian poset, and let $(\baX,\preceq)$ denote its completion.
Put
\begin{equation}
\label{eq:ome}
\Omega=\Delta(X)\sqcup\{\,(\eps,\eta)\in\der X\times \der X\mid \eps\preceq\eta\,\},
\end{equation}
where $\Delta(X)=\{\,(x,x)\mid x\in X\,\}$, and let $\sim$ denote the
equivalence relation on $\baX$ generated by the relation $\Omega$.
Then one has a canonical map 
\begin{equation}
\label{eq:projtilde}
\pi\colon\baX\to\tiX,
\end{equation}
where $\tiX=\baX/\!\!\sim$.
By construction, $\pi\vert_X$ is injective. 
The set $\tider X=\tiX\setminus \pi(X)$ will be called 
the {\em $\sim$-boundary} of the poset $(X, \preceq)$. 
We put 
\begin{equation}
I(\sim)=\{\,(\omega,\tau)\in \baX\times\baX \mid \omega\sim\tau \,\}\subseteq \baX\times\baX\\
\end{equation}
The set $\tiX$ carries the {\em quotient topology} $\caT_q(\tiX)$ with respect to the mapping 
$\pi$ and the topological space $(\baX,\caT_c(\baX))$.
In particular, the subspace $\pi(X)\subseteq\tiX$ is discrete and open, and
$\tider X\subseteq\tiX$ is closed.
For $\omega\in\baX$ we put $\bcC_\omega=\pi(\cC_\omega(\baX))$.
The space $\tiX$ will be considered merely as topological space.
It has the following property.

\begin{prop}
\label{prop:frech}
The topological space $(\tiX,\caT_q(\tiX))$ is a $T_1$-space.
\end{prop}
\begin{proof}
For $\omega\in\baX$ one has
\begin{equation}
\label{eq:propT1}
\caS(\pi(\omega))=\pi(\bigcap_{\tau\sim\omega}\caS(\tau))=
\pi(\bigcap_{\tau\sim\omega} \cC_\tau(\baX))=\{\pi(\omega)\}. 
\end{equation}
This yields the claim.
\end{proof}

\section{Monoids and their boundaries}
\label{s:monoids}
A {\em monoid} (or semigroup with unit) $\moM$ is a set with
an associative multiplication $\argu\cdot\argu\colon \moM\times \moM\to \moM$
and a distinguished element $1\in \moM$ satisfying
$1\cdot x=x\cdot1=x$ for all $x\in \moM$. For a monoid $\moM$ we denote by
\begin{equation}
\label{eq:maxgr}
\moM^{\times}=\{\,x\in \moM\mid \exists y\in \moM\colon x\cdot y=y\cdot x=1\,\}
\end{equation}
the {\em maximal subgroup} contained in $\moM$.


\subsection{$\N_0$-graded monoids}
\label{ss:gradmon}
The set of non-negative integers $\N_0=\{0,1,2,\ldots\}$
together with addition is a monoid.
A monoid $\moM$ together with a homomorphism of monoids
$|\argu|\colon \moM\to\N_0$ is called an {\em $\N_0$-graded monoid}.
For $k\in\N_0$ one defines $\moM_k=\{\,x\in \moM\mid |x|=k\,\}$. 
The $\N_0$-graded monoid $\moM$ is said to be {\em connected}, if $\moM_0=\{1\}$.
One has the following straightforward fact.

\begin{fact}
\label{fact:monvon}
For a connected $\N_0$-graded monoid $\moM$ one has
$\moM^\times=\{1\}$.
\end{fact} 
If $\moQ$ and $\moM$ are $\N_0$-graded monoids, a homomorphism $\phi\colon \moQ\to \moM$  is a
homomorphism of $\N_0$-graded monoids, if $\phi(\moQ_k)\subseteq \moM_k$
for all $k\in\N_0$. The following property is straightforward.

\begin{prop}
\label{prop:homs}
Let $\phi\colon\moQ\to\moM$ be a homomorphism of $\N_0$-graded monoids.
Then $\phi$ is monoton, i.e.,  $x,y\in\moQ$, $x\preceq y$ implies
$\phi(x)\preceq\phi(y)$, and thus induces a monoton map
\begin{equation}
\label{eq:bphi}
\cuD\phi\colon\cuD(\N,\moQ,\preceq)\longrightarrow\cuD(\N,\moM,\preceq).
\end{equation}
Let $\bphi\colon\bmoQ\to\bmoM$ denote the induced map.
Let $\partial\bphi\colon\partial\moQ\longrightarrow\partial\moM$ be the map induced by $\bphi$.
Then $\partial\bphi$ is continuous with repect to the cone topology.
\end{prop}

\begin{proof}
Let $\tau\in\moM$. Then the monotony of $\cuD\phi$ implies that
\begin{equation}
\label{eq:conthom}
\bphi^{-1}(C_\tau(\bmoM))=\bigcup_{y\in \euS} C_y(\bmoQ),
\end{equation}
where $\euS=\{\,q\in\bmoQ\mid \bphi(q)\in C_\tau(\bmoM)\,\}$. Thus $\bphi$
and  $\partial\phi$ are  continuous.
\end{proof}


\subsection{$1$-generated monoids}
\label{ss:1gen}
For any set $Y$ there exists a {\em free monoid} $\euF\langle Y\rangle$
which is naturally $\N_0$-graded. Moreover, $\euF\langle Y\rangle$ is connected
and $\euF\langle Y\rangle_1=Y$. For an $\N_0$-graded monoid $\moM$ there exists
a canonical homomorphism of $\N_0$-graded monoids
\begin{equation}
\label{eq:canmon}
\phi_\moM\colon \euF\langle \moM_1\rangle\longrightarrow \moM
\end{equation}
satisfying $\phi_{\moM,1}=\iid_{\moM_1}$. The $\N_0$-graded monoid $\moM$ is said to be
{\em $1$-generated}, if $\phi_\moM$ is surjective. In particular,
such a monoid is connected. By definition, free monoids are $1$-generated.
Moreover, $\moM$ is said to be {\em finitely $1$-generated}, if it is $1$-generated
and $\moM_1$ is a finite set. The following important question remains unanswered in this paper.

\begin{ques}
\label{ques:comp} Does there exist a finitely $1$-generated monoid $\moM$ satisfying
$\caT_c(\bmoM)\not=\caT_f(\bmoM)$.
\end{ques}


\subsection{Monoids as posets}
\label{ss:monpos}
Let $\moM$ be a monoid. 
For $x\in \moM$, put
\begin{align}
\moM x&=\{\,yx\mid y\in \moM \,\};\\
x\moM&=\{\,xy\mid y\in \moM \,\}.
\end{align}
For $x,y\in \moM$ we define
\begin{equation}
x\preceq y \Longleftrightarrow x\moM\subseteq y\moM,
\end{equation}
i.e.,  $x\preceq y$ if, and only if, there exists $z\in \moM$ such that $x=yz$.


\subsection{Left cancellative monoids}
\label{ss:lcmon}
A monoid $\moM$ is said to be {\it left cancellative} if $xy=xz$ implies $y=z$
for all $x,y,z\in \moM$; and {\it right cancellative} if $yx=zx$ implies $y=z$
for all $x,y,z\in \moM$.

\begin{prop} 
Let $\moM$ be a left-cancellative monoid.
For $x,y\in \moM$ one has $\moM x=\moM y$ if, and only if, 
there exists $z\in \moM^\times$ such that $y=xz$.
\end{prop}
\begin{proof}
For $z\in \moM^{\times}$ one has $z\moM=\moM$.
Thus for $x\in\moM$ and $y=xz$, multiplying by $x$ from the left yields $y\moM=x\moM$.
Viceversa, suppose $x\moM=y\moM$ for $x,y$ in $\moM$.
Then there exist $z,w\in \moM$ such that $y=xz$ and $x=yw$.
Hence $y=ywz$ and $x=xzw$.
Thus left cancellation implies $zw=1=wz$, showing that $z,w\in \moM^\times$.
\end{proof}

\begin{cor}
\label{cor:canmon}
Let $\moM$ be a left-cancellative monoid. Then $(\moM/\moM^\times,\preceq)$ is a poset.
\end{cor}

\begin{rem}
If left cancellation is replaced by right cancellation, then one has
$x\moM=y\moM$ if, and only if, there exists $z\in \moM^\times$ such that $y=zx$.
\end{rem}


\subsection{$1$-generated monoids as posets}
\label{ss:1genpos}

\begin{prop}
\label{prop:poset}
Let $\moM$ be a connected $\N_0$-graded monoid.
For $x,y\in \moM$ one has $x\moM=y\moM$ if, and only if, $x=y$.
\end{prop}
\begin{proof}
Suppose $x\moM=y\moM$, for $x,y\in \moM$.
Then there exist $z,w\in \moM$ such that $x=yz$ and $y=xw$,
so $|x|=|y|+|z|$ and $|y|=|x|+|w|$. 
Thus $|z|=0=|w|$. 
Since $\moM$ is connected, this implies $z=1=w$.
\end{proof}

As a consequence one obtains the following.

\begin{cor}
\label{cor:poset}
Let $\moM$ be a $1$-generated $\N_0$-graded monoid. 
Then $(\moM,\preceq)$ is a poset. If $\moM$ be finitely $1$-generated,
then $(\moM,\preceq)$ is a noetherian poset.
\end{cor}

\begin{rem}
\label{rem:tilno}
The following example shows that the universal boundary $\partial\moM$
is in general different from the $\sim$-boundary $\wtM$.
\label{ex:stup}
Let $\moM=\langle x,y,z \mid xz=zx\,\rangle$. Consider
\begin{equation}
\label{eq:ex1}
\begin{aligned}
f_1\colon& \N\to M,\qquad f_1(n)=(xz)^n,\\
f_2\colon& \N\to M,\qquad f_2(n)=x^n,\\
f_3\colon& \N\to M,\qquad f_3(n)=z^n.
\end{aligned}
\end{equation}
Then $f_2\succeq f_1\preceq f_3$. Hence $\pi(f_1)=\pi(f_2)=\pi(f_3)\in\wtM$,
and $\pi\colon\partial M\to\wtM$ is not injective.
\end{rem}

\subsection{Abelian semigroups generated by idempotents}
\label{ss:sgrpid}
Let $E$ be an abelian semigroup being generated by a set of 
elements $\Sigma\subseteq E$ satisfying $\sigma^2=\sigma$ for
all $\sigma\in\Sigma$, i.e., all elements of $\Sigma$ are idempotents.
Then every element $u\in E$ is an idempotent, and one may define
a partial order "$\preceq$" on $E$ by
\begin{equation}
\label{eq:pord}
u\preceq v \qquad \Longleftrightarrow \qquad u\cdot v=v,
\end{equation}
for $u,v\in E$. Let $\caR=\{\,(u,v)\in \Sigma\times \Sigma\mid u\preceq v\,\}$.
By definition, one has
\begin{equation}
\label{eq:E1}
E=\{\,u=\sigma_1\cdots\sigma_r\mid \sigma_i\in\Sigma\,\}.
\end{equation}
Hence
\begin{equation}
\label{eq:E2}
E\simeq \euF^{\ab}(\Sigma)/R,
\end{equation}
where $\euF^{\ab}(\Sigma)$ is the free abelian semigroup over the set $\Sigma$,
and $R$ is the relation
\begin{equation}
\label{eq:E3}
R=\{\, (uv,v) \mid (u,v)\in\caR\,\}\subseteq \euF^{\ab}(\Sigma)\times \euF^{\ab}(\Sigma),
\end{equation}
i.e., $E=\euF^{\ab}(\Sigma)/R^\sim$, where $R^\sim$ is the equivalence relation on
$\euF^{\ab}(\Sigma)$ generated by the set $R$.
Let
\begin{equation}
\whE=\{\,\chi\colon E\to\{0,1\}\mid \text{$\chi$ a semigroup homomorphism, $\chi(0)=0$, $\chi\not\equiv 0$}
\,\}
\end{equation}
Then $\whE$ coincides with the set of characters of the $C^\ast$-algebra $C^\ast(E)$
generated by $E$ (satisfying $e^\ast=e$ for all $e\in E$), and hence carries naturally
the structure of a compact topological space. By construction, $\whE$ can be identified with 
a subset of $\caF(\Sigma,\{0,1\})$ - the set of functions from $\Sigma$ to $\{0,1\}$. In more detail,
\begin{equation}
\label{eq:ident1}
\whE=\{\,\phi\in\caF(\Sigma,\{0,1\})\mid \forall (u,v)\in\caR:\ \phi(v)=\phi(u)\cdot\phi(v)\,\}
\end{equation}
Thus identifying $\euF(\Sigma,\{0,1\})$ with $\{0,1\}^{\Sigma}$, one obtains that
\begin{equation}
\label{eq:ident2}
\whE=\{\,(\eta_\sigma)_{\sigma\in\Sigma}\in\{0,1\}^{\Sigma}\mid\forall (u,v)\in\caR:
\sigma_v=\sigma_u\cdot\sigma_v\,\}.
\end{equation}


\subsection{The semigroup of idempotents generated by a set of subset of a set}
\label{ss:semiset}
Let $X$ be a set, and let $S\subseteq\euP(X)$ be a set of subsets of $X$.
Then $S$ generates an algebra of sets $\caA(S)\subseteq\euP(X)$, i.e., the sets of $\caA(S)$
consist of the finite intersections of sets in $S$. Then
\begin{equation}
\label{eq:semiset}
E(S)=\langle\,I_A\mid A\in\caA(S)\,\rangle\subseteq\euF(X,\{0,1\})
\end{equation}
is an abelian semigroup being generated by
the set of idempotents
\begin{equation}
\label{eq:semiset2}
\Sigma=\{\,I_Y\mid Y\in S\,\}.
\end{equation}
Moreover, by \eqref{eq:ident1}, one has
\begin{equation}
\label{eq:semiset3}
\whE(S)=\{\,\phi\in\caF(S,\{0,1\})\mid \forall U,V\in S, V\subseteq U: \phi(V)=
\phi(U)\cdot \phi(V)\,\}
\end{equation}


\subsection{The Laca-boundary of a monoid}
\label{ss:lacab}
Let $\moM$ be a $1$-generated monoid. Then one chooses 
\begin{equation}
\label{eq:laca1}
S=\{\, \omega\cdot \moM\mid\omega\in \moM\,\}
\end{equation}
to consist of all principal right ideals.
For short we call the compact set $\eth \moM=\whE(S)$ for $S$ as in \eqref{eq:laca1}
the {\emph{Laca boundary}} of $\moM$. 
For an infinite word $\uomega=(\omega_k)\in\cuD(\N,\moM,\succeq)$ one defines the element
$\chi_{\uomega}\in\whE(S)$ by $\chi_{\uomega}(\tau \moM)=1$ if, and only if,
there exists $k\in\N$ such that $\omega_k\in\tau \moM$, i.e., $\tau\succeq\omega_k$, and thus
$\tau\succeq\uomega$. This yields a map
\begin{equation}
\label{eq:map1}
\chi_{\cdot}\colon \cuD(\N,\moM,\preceq)\longrightarrow \whE(S)
\end{equation}
(cf. \cite[\S~2.2]{Lietal:Cstar}). By definition, it has the following property:

\begin{prop}
\label{prop:val}
For $\uomega=(\omega_k)\in\cuD(\N,\moM,\preceq), \tau\in \moM$, one has
$\chi_{\uomega}(\tau \moM)=1$ if, and only if, $\tau\succeq\uomega$.
In particular, one has $\chi_{\ueta}=\chi_{\uomega}$ if, and only if, $\ueta\approx\uomega$, and hence
$\chi_\cdot$ induces an injective map
\begin{equation}
\label{eq:map2}
\bchi_\cdot\colon \partial \moM\longrightarrow\eth \moM.
\end{equation}
\end{prop}

\begin{proof}
The first part has already been established before. 
Let $\ueta=(\eta_k)$. Then by the first part, $\uomega\succeq\ueta$ implies that
for all $\tau\in\moM$ one has
\begin{equation}
\label{eq:map3}
\chi_{\uomega}(\tau\moM)=1\Longrightarrow \chi_{\ueta}(\tau\moM)=1.
\end{equation}
Thus as $\image(\chi_{\omega})\subseteq\{0,1\}$ one concludes that
$\uomega\succeq\ueta$ and $\uomega\preceq\ueta$ implies that $\chi_{\uomega}=\chi_{\ueta}$.
On the other hand $\chi_{\ueta}=\chi_{\uomega}$ implies that  $1=\chi_{\ueta}(\eta_k \moM)=\chi_{\uomega}(\eta_k \moM)$ for all $k\in\N$.
In particular, $\ueta\succeq\uomega$. Interchanging the roles of
$\ueta$ and $\uomega$ yields $\uomega\succeq\ueta$, and thus $\ueta\approx\uomega$ (cf. section~\ref{ss:compl}).
The last part is a direct consequence of the definition of $\partial \moM$.
\end{proof}

The following theorem shows that for a $1$-generated $\N_0$-graded monoid $\moM$
its universal boundary
$\partial\moM$ with the fine topology is a totally-disconnected compact space.

\begin{thm}
\label{thm:princ}
 The map $\bchi_{\cdot}\colon (\partial \moM,\caT_f(\bmoM))\longrightarrow \eth \moM$ is a homeomorphism.
\end{thm}

\begin{proof}
It is well known that $\chi$ is surjective (cf. \cite[Lemma~2.3]{Lietal:Cstar}), and thus $\bchi$ is surjective. By Proposition~\ref{prop:val}, $\bchi$ is injective, and hence $\bchi$ is a bijection. The sets
\begin{equation}
\label{eq:subbase}
U_\tau^{\eps}=\{\,\eta\in\whE(\moM)\mid\eta(\tau\moM)=\eps\,\},\qquad \tau\in\moM,\qquad\eps\in\{0,1\}
\end{equation}
form a subbasis of the topology of $\whE(\moM)$, and 
\begin{equation}
\label{eq:map4}
\bchi^{-1}(U_\tau^1)=C_\tau(\bar{\moM}))\cap\partial\moM
\end{equation}
by \eqref{eq:map3}. Hence
$\bchi^{-1}(U_\tau^0)=C_\tau(\bar{\moM}))^C\cap\partial\moM$ and this yields the claim.
\end{proof}

The proof of Theorem~\ref{thm:princ} has also shown that
\begin{equation}
\label{eq:chiback}
\bchi^{-1}\colon\eth\moM\longrightarrow (\partial\moM,\caT_c(\bmoM))
\end{equation}
is a bijective and continuous map. This has the following consequence
(cf. \cite[\S~ 9.4, Corollary 2]{bou:top}).

\begin{prop}
\label{prop:T2}
Let $\moM$ be a $1$-generated $\N_0$-graded monoid such that
$(\partial\moM,\caT_c(\bmoM))$ is Hausdorff. Then $\moM$ is $\caT$-regular.
\end{prop}

In contrast to Proposition~\ref{prop:homs} one has the following property for the
Laca boundary of monoids.

\begin{prop}
\label{prop:hom}
Let $\phi\colon\moQ\longrightarrow\moM$ be a surjective homomorphism of connected
$\N_0$-graded monoids. Then $\phi$ induces an injective continuous map
$\phi_{\eth}\colon\eth\moM\longrightarrow\eth\moQ$.
\end{prop}

\begin{proof}
By Proposition~\ref{prop:poset}, $\phi$ induces a map $\phi_{\Sigma}\colon\Sigma(\moQ)\to\Sigma(\moM)$ given by
\begin{equation}
\label{eq:phiS}
\phi_{\Sigma}(\omega\moQ)=\phi(\omega)\moM.
\end{equation}
Moreover, for $x,y\in\moQ$ one has $x\preceq y$, if and only if, $x\moQ\subseteq y\moQ$,
if and only if there exists $z\in\moQ$ such that $x=y\cdot z$.
From the last statement one concludes that $\phi_\Sigma(x\moQ)\subseteq
\phi_{\Sigma}(z\moQ)$. Thus, by \eqref{eq:E2}, $\phi_\Sigma$ induces a homomorphism of 
semigroups
\begin{equation}
\label{eq:phiE}
\phi_E\colon E(\moQ)\longrightarrow E(\moM),
\end{equation}
and thus a map
\begin{equation}
\label{eq:hEphio}
\phi_{\whE}^\circ\colon \whE(\moM)\cup\{0\}\longrightarrow \whE(\moQ)\cup\{0\}.
\end{equation}
If $\phi$ is surjective, then $\phi_E$ is surjective, and $\phi_{\whE}^\circ$ restricts
to a map 
\begin{equation}
\label{eq:hEphi}
\phi_{\whE}\colon\whE(\moM)\longrightarrow\whE(\moQ). 
\end{equation}
It is straightforward to verify that $\phi_{\whE}$ is continuous and injective.
\end{proof}

\section{Free monoids and trees}
\label{ss:freetree}
Let $\euF_n=\euF\langle x_1,\dots,x_n \rangle$ be the free monoid on $n$ generators.
Let $S=\{ x_1,\dots,x_n \}$ be the set of generators, and let $|\argu|\colon\euF_n\to\N_0$
be the grading morphisms, i.e., $|y|=1$ if and only if $y\in S$.
Then the Cayley graph $\Gamma(\euF_n, S)$ of $\euF_n$ is the graph defined by
\begin{align}
V&= \{\, x\mid x\in \euF_n\,\}\\
E&=\big\{\, (x,xx_i)\in V\times V\mid x\in \euF_n, x_i\in S \,\big\}.
\end{align}
One has two maps $o,t\colon E \to V$ defined by
\begin{align}
o\big((x,xx_i)\big)&=x;\\
t\big((x,xx_i)\big)&=xx_i.
\end{align}
Then $\Gamma(\euF_n, S)$ is an $n$-regular tree with root $1$ and all edges
pointing away from $1$. The graph $\Gamma(\euF_n,S)$ is not a graph in the sense
of  \cite[\S~2.1]{ser:trees}), but coincides with an orientation on the $n$-regular tree $T_n$.

\subsection{The boundary of the $n$-regular tree}
\label{ss:boundT}
The boundary $\der T_n$ of $T_n$ is the set of equivalence classes 
of infinite paths without backtracking under the relation $\sim$
defined by the shift, i.e. 
\begin{equation}
v_0 v_1 v_2\dots\sim v_1 v_2 v_3\dots
\end{equation}
We denote by $[v,w)$ the unique path starting at $v$ in the class $\omega$
and define 
\begin{equation}
I_v=\{\,\omega\in\der T_n\mid v\in[1,\omega) \,\}
\end{equation}
the {\it interval} of $\der T_n$ starting at $v$.
Then $\der T_n$ is compact with respect to the topology $\caT_I$ generated by $\{ I_v\}_{v\in V}$
(cf. \cite[p.~20, Exercise~1]{ser:trees}).

For any $[\rho]\in\der T_n$ there exists a unique ray 
$\rho=(e_k)_{k\in\N}$, $o(\rho)=o(e_1)=1$.
One can assign to $\rho$ the decreasing function 
$\omega_\rho\in \cuD(\N,\euF_n,\preceq)$ given by
$\omega_\rho(k)=t(e_k)$.
The map $\varphi\colon\der T_n\to\partial\euF_n$ given by 
\begin{equation}
\label{boundtree}
\varphi ([\rho])= [\omega_\rho]
\end{equation}
is a bijection. Hence one can identify $\der T_n$ with $\partial\euF_n$.

\subsection{The space $(\beuF_n,\caT_c(\beuF_n))$}
\label{ss:comp}
Every cone $C_\tau(\beuF_n)$ defines a rooted subtree $ T_\tau$ of $T_n$ satisfying
$\der T_\tau=\partial\euF_n\cap C_\tau(\beuF_n)$. Thus every covering
$\bigcup_{\tau\in  U} C_\tau(\beuF_n)\cap\partial\euF_n$ of the boundary
of $\der\euF_n$ by cones defines a forest $F=\bigcup_{\tau\in U} T_\tau$.
Let $F=\bigcup_{i\in I} F_i$ be the decomposition of $F$ in connected components.
Then $\der T_n=\der F=\bigsqcup_{i\in I} \der F_i$, where $\sqcup$ denotes disjoint union.
Hence the compactness of $\der T_n$ implies $|I|<\infty$. 

As $\der F_i\subseteq \der T_n$ is closed, and hence compact, a similar argument shows that there exist finitely many cones $C_{\tau_{i,j}}$, $1\leq j\leq r_i$, such that
$F_i=\bigcup_{1\leq j\leq r_i} T_{\tau_{i,j}}$.
Thus, if $\bigcup V$ is an open covering of $\beuF_n$ by open sets, it can be refined
to a covering $\bigcup U$, where $U$ consists either of a cone $C_\tau(\beuF_n)$ or of a
singleton set $\{\omega\}$, $\omega\in \euF_n$.
Let $\Lambda\subseteq T_n$ be the subtree being generated by the vertices $\tau_{i,j}$.
Then $\Lambda$ is a finite subtree, and the only vertices of $T_n$ not being covered by
$\bigcup_{i,j} C_{\tau_{i,j}}(\beuF_n))$ are contained in $V(\Lambda)$. This shows that
$(\beuF_n,\caT_c(\beuF_n))$ is a compact space.

\subsection{The space $(\bmoM,\caT_c(\bmoM))$}
Let $\moM$ be a $\caT$-regular  finitely $1$-generated monoid.  Then, by definition,
$(\der\moM,\caT_c(\bmoM))$ is a Hausdorff space, and hence  
$(\bmoM,\caT_c(\bmoM))$ is a Hausdorff space. By Proposition \ref{prop:homs},
the canonical mapping $\phi_\moM\colon\euF\to\moM$ (cf. \eqref{eq:canmon}) induces
a continuous surjective map $\bphi_{\moM}\colon\beuF\to\bmoM$. This shows the following.

\begin{prop}
\label{prop:Treg}
Let $\moM$ be a finitely $1$-generated $\N_0$-graded $\caT$-regular monoid.
Then $(\bmoM,\caT_c(\bmoM))$ is a compact space.
\end{prop}

\subsection{The canonical probability measure on the boundary of a regular tree}
\label{ss:cammu}
By Carath\'eodory's extension theorem the assignment
\begin{equation}
\label{eq:mudTn}
\mu(I_v)=n^{-|v|}
\end{equation}
defines a unique probability measure $\mu\colon\Bor(\der T_n)\to\R^+_0$. 
Hence the corresponding probability measure $\mu\colon\Bor(\partial{\euF}_n)\to\R^+_0$
satisfies
\begin{equation}
\label{eq:mupaM}
\mu(\partial\euF_n\cap C_\tau(\bar{\euF}_n))= n^{-|\tau|}\ \text{for $\tau\in\euF_n$}.
\end{equation}

\begin{defi}
Let $\argu\cdot\argu\colon \euF_n\times\partial \euF_n\to\partial \euF_n$ be the map given by
\begin{equation}
x\cdot[\omega]=[x\omega],
\end{equation}
where $x\omega\colon\N\to \euF_n$ is given by $(x\omega)(n)=x\omega(n)$
\end{defi}
Note that this action is well defined, since $\omega\sim\omega^\prime$ implies that
$x\omega\sim x\omega^\prime$.

\begin{defi}
Let $\argu\cdot\argu\colon L^2(\partial \euF_n,\C,\mu)\times \euF_n\to L^2(\partial\euF_n,\C,\mu)$ 
be the map given by
\begin{equation}
\label{eq:highx}
f.x={}^xf,
\end{equation}
where
\begin{equation}
({}^xf)([\omega])=f([x\omega]).
\end{equation}
\end{defi}
Note that for $f\in L^2(\partial\euF_n,\C,\mu)$ one has ${}^xf\in L^2(\partial\euF_n,\C,\mu)$,
since
\begin{align}
\|\ {}^xf \|_2^2
&=\int_{\der T_n} |\ {}^xf|^2\, d\mu\\
&=\int_{x\der T_n} |f|^2\, d\mu\\
\label{eqmeas}
&\le\int_{\der T_n} |f|^2\, d\mu\\
&=\| f\|_2^2,
\end{align}
where \eqref{eqmeas} follows 
since $x \partial\euF_n\subseteq \partial \euF_n$.

\begin{defi}
For $z\in \euF_n$ we define the map 
$T_z\colon L^2(\partial\euF_n,\C,\mu)\to L^2(\partial\euF_n,\C,\mu)$ by
\begin{equation}
\label{eq:defTx}
T_z(f)= {}^zf.
\end{equation}
\end{defi}

\begin{fact}
\label{fact:fact1}
$\euF_n$ acts via $T_\cdot$ on $L^2(\partial\euF_n,\C,\mu)$ by bounded linear operators.
\end{fact}
\begin{proof}
Let $z\in \euF_n$. For $f,g\in L^2(\partial\euF_n,\C,\mu)$, $[\omega]\in\partial\euF_n$, one has
\begin{align}
\big(T_z(f+g)\big)([\omega])
&=(f+g)([z\omega])\\
&=f([z\omega])+g([z\omega])\notag\\
&=( {}^zf)([\omega])+( {}^zg)([\omega])\notag\\
&=\big(T_z (f)\big) ([\omega])+ \big(T_z (g)\big) ([\omega]).\notag
\end{align}
Thus $T_z$ is linear.
It is also bounded, since
\begin{equation}
\| T_z\|_\infty
= \sup_{\| f\|_2=1} \|T_z (f)\|_2
\le  \sup_{\| f\|_2=1} \|f\|_2\le1.
\end{equation}
\end{proof}

Hence $T_z\in\mathcal{B}\big(L^2(\partial\euF_n,\C,\mu)\big)$ for all $z\in \euF_n$.
As  $\mathcal{B}\big(L^2(\partial\euF_n,\C,\mu)\big)$ is a $C^\ast$-algebra,
$T_z$ has an adjoint operator $T_z^\ast$,
which is the bounded operator satisfying
\begin{equation}
\langle\, T_z f,g\,\rangle =\langle\, f,T_z^\ast g\,\rangle,
\end{equation}
for all $f,g\in L^2(\partial\euF_n,\C,\mu)$.

\begin{fact}
\label{fact:adj}
The bounded operator $T_z^\ast$, for $z\in \euF_n$, is given by 
\begin{equation}
\label{eq:star}
(T_z^\ast f) ([\omega])=
\begin{cases}
0 &\mbox{ if } [\omega]\notin z\partial\euF_n\\
f([\omega^\prime]) &\mbox{ if } [\omega]=x[\omega^\prime].
\end{cases}
\end{equation}
\end{fact}
\begin{proof}
Note that $T_z^\ast f\in L^2(\partial\euF_n,\C,\mu)$, since
\begin{align}
	\| T_z^\ast f\|_2^2
	&=\int_{\partial\euF_n} |T_z^\ast f|^2\, d\mu\\
	&= \int_{x\partial\euF_n} |T_z^\ast f|^2\, d\mu\\
	&\le \int_{\der T_n} |f|^2\, d\mu.
\end{align}
Let $f,g\in L^2(\partial\euF_n,\C,\mu)$. Then one has
\begin{align}
	\langle\,f, T_z^\ast g \,\rangle
	&=\int_{\der T_n} f(\overline{T_z^\ast g}) \,d\mu\\
	&=\int_{z\der T_n} f(\overline{T_z^\ast g}) \,d\mu\\
	\label{adj}
	&=\int_{z\der T_n} (T_z f)\overline{g} \,d\mu\\
	&\le\int_{\der T_n} (T_z f)\overline{g} \,d\mu\\
	&=\langle\,T_z f, g \,\rangle.
\end{align}
where equality \eqref{adj} holds by 
\begin{equation}
f([z\omega^\prime])\,\overline{T_z^\ast g}([z\omega^\prime])
=(T_z f)([\omega^\prime])\,\bar{g}([\omega^\prime]).
\end{equation}
\end{proof}

\begin{prop}
\label{prop:cuntz}
The following identities hold for all $x,y\in S\subseteq\euF_n$
\begin{align}
	\label{eqdelta}
	T_x^\ast T_y&=\delta_{xy};\\
	\label{eqsum}
	\sum_{i=1}^n T_{x_i} T_{x_i}^\ast &=1.
\end{align}
In particular, the $C^\ast$-algebra 
$C^\ast(\euF_n,\mu)\subseteq\caB(L^2(\partial\euF_n,\C,\mu))$ 
generated by $\euF_n$ is isomorphic to the Cuntz algebra $\mathcal{O}_n$.
\end{prop}
\begin{proof}
Let $x,y\in S\subseteq \euF_n$ and let $f\in L^2(\partial\euF_n,\C,\mu)$.
For any $[\omega]\in\partial\euF_n$ one has
\begin{equation}
T_x^\ast T_y(f)([\omega])= T_x^\ast f([y\omega]) = \delta_{xy}
\end{equation}
by Fact \ref{fact:adj}. This proves identity \eqref{eqdelta}.

Let $[\omega]\in\partial\euF_n$. 
Then there exists $x_j\in S$ such that $[\omega]\in x_j\partial\euF_n$.
Hence one has
\begin{equation}
T_{x_i}T_{x_i}^\ast f ([\omega])=\delta_{ij} f([\omega])
\end{equation}
for any $f\in L^2(\partial\euF_n,\C,\mu)$.
This yields the identity \eqref{eqsum}.
\end{proof}

\subsection{Finitely $1$-generated monoids}
\label{ss:caT}
Let $\moM$ be a finitely $1$-generated $\N_0$-graded monoid.
Then one has a canonical surjective graded homomorphism $\phi_\moM\colon\euF\to\moM$,
where $\euF$ is a finitely generated free monoid (cf. \eqref{eq:canmon}),
which induces a continuous map $\partial\phi\colon\partial\euF\longrightarrow\partial\moM$
(cf. Proposition~\ref{prop:homs}).
In particular, 
\begin{equation}
\mu_\moM\colon\Bor(\partial\moM)\longrightarrow\R^+_0
\end{equation} 
given by $\mu_\moM(A)=\mu(\partial\phi_\moM)^{-1}(A))$
is a Borel probability measure on $\partial\moM$.

For $s\in\moM$, define the map $\beta_s\colon\der\moM\to\der\moM$ by 
\begin{equation}
\label{eq:act1}
	\beta_s([f])=[sf],\quad [f]\in\der\moM,
\end{equation}
where $(sf)(n)=s\cdot f(n)$ for all $n\in\N$, $f\in\cuD(\N,\moM,\preceq)$.
Then, as $\beta_s$ is mapping cones to cones, $\beta_s$ is continuous.
Hence one has a representation
\begin{equation}
\label{eq:act2}
	\beta\colon \moM\to C(\der \moM,\der \moM).
\end{equation}
For $s\in\moM$, let $\beta_{\ast,s}\colon L^2(\der\moM,\C,\mu)\to L^2(\der\moM,\C,\mu)$
be the map defined by
\begin{equation}
\label{eq:act3}
	\beta_{\ast,s}(g)([f])=g(\beta_s([f]))=g([sf]),\quad\,g\in L^2(\der\moM,\mu),\,[f]\in\der\moM.
\end{equation}
Then one has 
\begin{align}
	\|\beta_{\ast,s}(g)\|_2^2
	&=\int_{\der\moM} \big|g\big(\beta_s([f]\big)\big|^2\, d\mu_{\moM}\\
	&=\int_{\der\moM} \big|g\big([sf]\big)\big|^2 \, d\mu_{\moM}\\
	&=\int_{s\der\moM} \big|g\big([f]\big)\big|^2 \, d\mu_{\moM}\\
	&\le\int_{\der\moM}\big| g \big([f]\big) \big|^2 \,d\mu_{\moM}\\
	&=\|g\|_2^2,
\end{align}
for all $g\in L^2(\der\moM,\C,\mu)$, $s\in \moM$.
Thus,
\begin{equation}
\label{eq:act4}
	\|\beta_{\ast,s}\|=\sup_{\| g\|_2=1} \|\beta_{\ast,s}(g)\|_2\le 1
\end{equation}
for all $s\in\moM$, i.e., $\beta_{\ast,s}$ is a bounded operator on $L^2(\der\moM,\C,\mu)$.
By an argument similar to the one used in the proof of Fact~\ref{fact:fact1} one can show that
it is also linear. 
In particular, there exists a representation
\begin{equation}
\label{eq:act5}
	\beta_\ast\colon\moM\to\mathcal{B}\big(L^2(\der\moM,\C,\mu)\big).
\end{equation}


\subsection{Right-angled Artin monoids}
\label{ss:raam}
Let $\Gamma=(V,E)$ be a finite undirected graph, i.e. $|V|=n<\infty$
and $E\subseteq \mathscr{P}_2(V)$ is a set of subsets of cardinality $2$ of $V$ . 
The right-angled Artin monoid associated to $\Gamma$ is the monoid $\moM^\Gamma$
defined by
\begin{equation}
\label{eq:raam1}
	\moM^\Gamma=\langle\,x\in V \mid xy=yx\mbox{ if } \{x,y\}\in E\,\rangle^+.
\end{equation}
Clearly, $\moM^\Gamma$ is $\N_0$-graded and finitely $1$-generated.
By Luis Paris theorem (cf. \cite{paris:artin}), $\moM^\Gamma$ embeds into the right-angled Artin group $G_\Gamma$. Thus  $\moM^\Gamma$ has  the left-cancellation
property as well as the right-cancellation property.
The canonical homomorphism $\phi_\Gamma\colon \euF\langle V\rangle\longrightarrow
\moM^\Gamma$ is surjetcive and induces a continuous surjective map
\begin{equation}
\label{eq:canraam}
\partial\phi_\Gamma\colon\partial\euF\langle V\rangle\longrightarrow\partial\moM^\Gamma.
\end{equation}
(cf. Proposition~\ref{prop:homs}). 
We denote by $\mu_\Gamma\colon\Bor(\partial\moM^\Gamma)\longrightarrow\R^+_0$
the Borel probability measure induced by $\partial\phi_\Gamma$, i.e., for
$A\in\Bor(\partial\moM^\Gamma)$ one has $\mu_\Gamma(A)=
\mu(\partial\phi_\Gamma^{-1}(A))$, where $\mu$ is the measure defined
on $\partial\euF\langle V\rangle$ by \eqref{eq:mupaM}.

\begin{defi}
	Let $\Gamma=(V,E)$ be a graph with unoriented edges,
	and let $\Gamma_1=(V_1,E_1)$ and $\Gamma_2=(V_2,E_2)$ be subgraphs
	of $\Gamma$. We say that $\Gamma$ is {\it bipartitly decomposed} by $\Gamma_1$
	and $\Gamma_2$, if $V=V_1\sqcup V_2$ and
	\begin{equation}
	\label{eq:bipar1}
	E=E_1\sqcup E_2\sqcup \big\{\,\{v_1,v_2\}\mid v_1\in V_1,\,v_2\in V_2\,\big\}.
	\end{equation}
	In this case we will write $\Gamma=\Gamma_1\vee\Gamma_2$.
	If no such decomposition exists, $\Gamma$ will be called {\it coconnected}.
\end{defi}

If $\Gamma$ is any graph which is not coconnected, 
we can decompose it into coconnected components.
We find these coconnected components by considering its
{\it opposite graph} $\,\Gamma^{\op}=(V,\mathscr{P}_2(V)\setminus E)$ and finding its connected components;
the coconnected components of $\Gamma$ correspond to 
the connected components of $\Gamma^\op$.

One has the following property.

\begin{fact}
	\label{fact:graph}
	Let $\Gamma=(V,E)$ be a graph with unoriented edges. Then $\Gamma$ is
	coconnected if, and only if, $\Gamma^{\op}$ is connected. In particular, if
	$\,\Gamma^{\op}=\bigsqcup_{i\in I} \Lambda_i$ is the decomposition of $\Gamma^{\op}$
	in its connected components, then one has
	\begin{equation}
	\label{eq:bipar2}
	\Gamma=\textstyle{\bigvee_{i\in I} \Lambda_i^{\op}},
	\end{equation}
	where $\Lambda_i^{\op}$ are coconnected subgraphs of $\Gamma$.
\end{fact}
\begin{proof}
	Obviously, the graph $\Gamma=\Gamma_1\vee\Gamma_2$ is bipartitly decomposed
	if, and only if, $\Gamma^{\op}=\Gamma_1^\op\sqcup\Gamma_2^\op$ is not connected.
	This yields to the claim.
\end{proof}

In analogy to the decomposition in connected components we will call \eqref{eq:bipar2}
the {\it decomposition in coconnected components}.
Note that the decomposition in coconnected components
implies that ant two vertices in different components 
must be connected by an edge. From this property one concludes
the following straightforward fact.

\begin{fact}
	\label{fact:coco1}
	Let $\Gamma=(V,E)$ be a finite graph with unoriented edges, and let
	$\Gamma=\bigvee_{1\leq i\leq r} \Gamma_i$ be its decomposition in coconnected
	components, $\Gamma_i=(V_i,E_i)$.Then
	\begin{equation}
	\label{eq:coco1}
	\moM^\Gamma=\moM^{\Gamma_1}\times\cdots\times \moM^{\Gamma_r},
	\end{equation}
	where $\moM^{\Gamma_i}=\langle\, v\in V_i\,\rangle$. In particular,
	$\partial\moM^\Gamma=\times_{1\leq r} \partial\moM^{\Gamma_i}$ and
	\begin{equation}
	\label{eq:prodL2}
	L^2(\partial\moM,\C,\mu)=L^2(\partial\moM^{\Gamma_r},\C,
	\mu_{\Gamma_1})\hotimes\cdots\hotimes 
	L^2(\partial\moM^{\Gamma_r},\C,\mu_{\Gamma_r}).
		\end{equation}
\end{fact}
In \cite{CL07}, J. Crisp and M.Laca has shown the following.

\begin{thm}[\cite{CL07}, Theorem 6.7]
	\label{th:CL07}
	Let $\Gamma=(V,E)$ be a finite unoriented graph such that $\Gamma^\op$ has no isolated vertices, and let $\Gamma=\bigvee_{i=1}^r \Gamma_i$ be the decomposition of $\Gamma$ in coconnected components, $\Gamma_i=(V_i,E_i)$. 
	Then the universal $C^\ast$-algebra 
	with generators $\{S_x\mid x\in V\}$ subject to the relations
	\begin{itemize}
		\item[(i)] $S_x^\ast S_x=1$ for each $x\in V$;
		\item[(ii)] $S_x S_y=S_y S_x$ and $S_x^\ast S_y=S_y S_x^\ast$ if $x$ and $y$ are adjacent in $\Gamma$;
		\item[(iii)] $S_x^\ast S_y=0$ if $x$ and $y$ are distinct and not adjacent in $\Gamma$;
		\item[(iv)] $\prod_{x\in V_i}(1-S_x S_x^\ast)=0$ for each $i\in\{1,\dots,r\}$;
	\end{itemize}
	is canonically isomorphic to the boundary quotient $\partial C_\lambda(\moM^\Gamma)$ for 
	$\moM^\Gamma$ and it is a simple $C^\ast$-algebra.
\end{thm}

Hence, one has the following proposition.
\begin{prop}
	The $C^\ast$-algebra $C^\ast(\moM^\Gamma,\mu_\Gamma)$ (cf. \eqref{eq:Cstarmon})
	 of a right-angled Artin monoid $\moM^\Gamma$
	is isomorphic to the boundary quotient $\partial C_\lambda(\moM^\Gamma)$ of Theorem \ref{th:CL07}.
\end{prop}
\begin{proof}
	Let $\Gamma=(V,E)$ be a finite unoriented graph such that $|V|=n$
	and let $\Gamma=\bigvee_{i=1}^r \Gamma_i$ be its decomposition in coconnected components.
	It is straightforward to verify that for the set of operators $\{T_x\mid x\in V\}$, where the 
	operator $T_x$ is defined as in \eqref{eq:defTx}, the adjoint operators are given by \eqref{eq:star},
	and thus $\{T_x\mid x\in V\}$ satisfies the relations (i)-(iii). 
	It remains to prove that it also satisfies (iv). Let 
	\begin{equation}
	\label{eq:defidem}
	\be_i=\prod_{x\in V_i}(1-T_x T_x^\ast).
	\end{equation}
	In order to show that $\be_i(f)=0$ for all $f\in L^2(\partial\moM,\C,
	\mu_{\Gamma})$ it suffices to show that 
	 $\be_i(f)=0$ for $f=f_1\otimes\cdots\otimes f_r$, $f_i\in 
	 L^2(\partial\moM_{\Gamma_i},\C,\mu_{\Gamma_i})$ (cf. \eqref{eq:prodL2}).
	Note that 
	\begin{equation}
	\label{eq:valprod}
	\big(1-T_xT_x^\ast\big)(f)([u])=
	\begin{cases}
	0 &\mbox{ if } [u]\in x\der \moM^\Gamma,\\
	f([u]) &\mbox{ otherwise.}
	\end{cases}
	\end{equation}
	Let $u=u_1\cdots u_r$, $u_j\in\partial\moM^{\Gamma_j}$. Then there
	exists $y\in V_i$ such that $u_i\in y\partial\moM^{\Gamma_i}$. Hence, by 
	\eqref{eq:valprod}
	\begin{equation}
	\big(1-T_y T_y^\ast\big)(f)([u])=0.
	\end{equation}
	Hence $\be_i(f)=0$ and this yields the claim.
	\end{proof}


\section{Fractals}
\label{ss:fract}
Let $\moM$ be a finitely $1$-generated monoid.
By an {\emph{$\moM$-fractal}} we will understand a complete metrix space $(X,d)$
with a contracting left $\moM$-action $\alpha\colon \moM\longrightarrow C(X,X)$, 
i.e., there exists a real number $\delta<1$ such that for all $x,y\in X$ and all $\omega\in
\moM\setminus\{1\}$ one has
\begin{equation}
\label{eq:cont1}
d(\alpha(\omega)(x),\alpha(\omega)(y))<\delta\cdot d(x,y).
\end{equation}
The real number $\delta$ will be called the contraction constant.
To the authors knowledge the following important question has not yet been discussed in the literature.

\begin{ques}
\label{ques:fract} For which finitely $1$-generated monoids $\moM$ does there exist an
$\moM$-fractal $(X,d,\alpha)$.
\end{ques}

\subsection{The action of the universal boundary on an $\moM$-fractal}
\label{ss:actuni}
For a strictly decreasing sequence $f\in\cuD(\N,\moM,\preceq)$ and for $n,m\in\N_0$, 
$m> n$,
there exists $\tau_{m,n}\in \moM\setminus\{1\}$ such that $f(m)=f(n)\cdot\tau_{m,n}$. 
By induction, one concludes
that $|f(n)|\geq n$. If $[f]\in\partial \moM$, then $f$ can be represented by a strictly
decreasing sequence (cf. Fact~\ref{fact:strict}).

As $\alpha$ is contracting, one concludes that $(\alpha(f_n)(x))$ is a Cauchy sequence
for every strictly decreasing sequence $f\in\cuD(\N,M,\preceq)$
and thus has a limit point $\alpha(f)(x)=\lim_{n\to\infty}(\alpha(f_n)(x))$. In more detail,
if $\alpha$ has contracting constant $\delta<1$,
 one has 
for $n,m\in\N$, $m> n$, that  
\begin{equation}
\label{eq:contt2}
d(\alpha(f(m))(x),\alpha(f(n))(x))<\delta^{|f(n)|}\cdot d(\alpha(\tau_{m,n})(x),x)\leq
\delta^{|f(n)|}\cdot\diam(X),
\end{equation}
where $\diam(X)=\max\{\,d(y,z)\mid x,y\in X\,\}$.
Thus one has a map
\begin{equation}
\label{eq:baMact}
\argu\cdot\argu\colon\cuD(\N,\moM,\prec)\times X\longrightarrow X
\end{equation}
given by $[f]\cdot x=\alpha(f)(x)$. This map has the following property.

\begin{prop}
\label{prop:frac}
Let $\moM$ be a finitely $1$-generated monoid, and let $((X,d),\alpha)$ be an
$\moM$-fractal with attractor $K\subseteq X$ . Then the map \eqref{eq:baMact} is continuous and $[f]\cdot x\in K$ for all $f\in\cuD(\N,\moM,\prec)$
and $x\in X$.
\end{prop}

\begin{proof}
Let $f\in\cuD(\N,X,\prec)$ be a 
strictly decreasing function.
For $A=\{x\}$, and $\euS(A)=\bigcup_{\sigma\in \moM_1} \sigma(A)$, the sequence
$\euS^k(A)$
converges to $K$ in the Hausdorff metric (cf. \cite[Statement (1)]{hutch:frac}). Thus 
for all $\eps>0$ there exists $N(\eps)\in\N$ such that for all $n>N(\eps)$
$\eud(\euS^n(A),K)<\eps$, where $\eud$ denotes the Hausdorff metric
(cf. \cite[(2.4)]{hutch:frac}). Hence $d(\alpha(f(n))(x),K)<\eps$ for all $n>N(\eps)$, and
$\alpha(f)(x)$ is a clusterpoint
of $K$. As $K$ is closed this implies $\alpha(f)(x)\in K$.

The map \eqref{eq:baMact} is obviously continuous in the second argument. Moreover, let $f,h\in\cuD(\N,X,\prec)$, $f, h\prec \tau$. Then
\begin{equation}
\label{eq:stetig1}
d(\alpha(f)(x)),\alpha(h)(x))\leq 2\cdot\delta^{|\tau|}\cdot\diam(X).
\end{equation}
Thus \eqref{eq:baMact} is continuous.
\end{proof}

\begin{prop}
\label{prop:leq}
Let $f,h\in\cuD(\N,\moM,\prec)$ satisfying $f\preceq h$. Then,
$\alpha(f)(x)=\alpha(h)(x)$.
\end{prop}

\begin{proof}
We may assume that $f(n)\preceq h(n)$  for all $n\in\N$, i.e., there exists $y_n\in \moM$
such that $f(n)=h(n)\cdot y_n$. Then, by the same argument which was used for \eqref{eq:contt2}, one concludes that
\begin{equation}
\label{eq:cont3}
d(\alpha(f_n)(x),\alpha(h(n)(x))\leq \delta^{|h(n)|}\diam(X)\leq\delta^n\diam(X).
\end{equation}
This yields the claim.
\end{proof}

From Proposition~\ref{prop:leq} one concludes that the map
\eqref{eq:baMact} induces an action
\begin{equation}
\label{eq:boact}
\argu\cdot\argu\colon\partial\moM\times X\longrightarrow X.
\end{equation}

\begin{prop}
\label{prop:uniprop}
Let $x\in X$, and let $K\subset X$ be the attractor of the $\moM$-fractal
$((X,d),\alpha)$. Then the induced map
\begin{equation}
\label{eq:contt4}
\kappa_x\colon\partial\moM\longrightarrow K
\end{equation}
given by $\kappa_x([f])=\alpha(f)(x)$ is surjective.
\end{prop}

\begin{proof}
Let $z\in K$, and $A=\{x\}$. By (cf. \cite[(2.4)]{hutch:frac}), for all $\eps>0$ there exists $N(\eps)\in\N$ such that for all $n>N(\eps)$
$\eud(\euS^n(A),z)<\eps$, i.e., there exists a sequence $(f_n)_{n\in\N}$, $f_n\in\moM_n$,
$f_{n+1}\in\bigcup_{\sigma\in\moM_1} \{\sigma\cdot f_n\}$, such that
$d(\alpha(f_n)(x),z)<\eps$. 

If $\moM$ is $\caT$-regular, then $(\bmoM,\caT_c(\bmoM))$ is compact (cf. Proposition~\ref{prop:Treg}). Hence $(f_n)_{n\in\N}$ has a cluster point $f\in\bmoM$.
As $|f_n|=n$, one has $f\not\in\moM$ and thus $f\in\partial\moM$. It is straightforward to verify that  $[f]\cdot x=z$, showing that $\kappa_x$ is surjective. 
\end{proof}

From Proposition~\ref{prop:leq} one concludes that the  $\partial \moM$-action on $X$ (cf. \eqref{eq:baMact}) induces a $\wtM$-action
\begin{equation}
\label{eq:wtMact}
\argu\cdot\argu\colon\wtM\times X\longrightarrow X
\end{equation}
given by $\pi([f])\cdot x=\alpha(f)(x)$ (cf. \eqref{eq:projtilde}).

\subsection{The $C^\ast$-algebra associated to an $\moM$-fractals for a finitely $1$-generated monoid $\moM$}
\label{ss:Mfrac}
Let $\moM$ be a finitely $1$-generated monoid, and let $((X,d),\alpha)$
be an $\moM$-fractal with atrractor $K$. For $x\in X$ there exists a continuous
mapping $\kappa_x\colon\der\moM\to K$ (cf. Theorem C). Let
$\mu_x\colon\Bor(K)\to\R^+_0$ be the probability measure given by
\eqref{eq:kont}. Then $\moM$ acts on $K$, and thus also on $L^2(K,\C,\mu_x)$.

For $t\in\moM$  let $\gamma_t\colon L^2(K,\C,\mu_x)\to L^2(K,\C,\mu_x)$ be given by
\begin{equation}
	\gamma_t(g)(x)=g(\alpha_t(x))
\end{equation}
where $g\in L^2(K,\C,\mu_x)$.
Hence the monoid $\moM$ acts on the Hilbert space 
$L^2(K,\C,\mu_x)$ by bounded linear operators.
\begin{equation*}
	\|\gamma_t(g)\|_2^2=\int_K |\gamma_t(g(z))|^2\,d\mu_x
	=\int_K|g(\alpha_t(z))|^2\,d\mu_x\le\|g\|_2^2
\end{equation*}
(cf. \S~\ref{ss:boundT}). One defines the $C^\ast$-algebra generated
by the $\moM$-fractal $((X,d),\alpha)$ by
\begin{equation}
\label{eq:defCstarfrak}
C^\ast(\moM,X,d,\mu_x)=\langle\,\gamma_t,\,\gamma_t^\ast\mid t\in\moM\,\rangle\subseteq
\caB(L^2(K,\C,\mu_x)).
\end{equation}



\bibliography{monoid}
\bibliographystyle{amsplain}

\end{document}